\documentclass{amsart}

\usepackage{amsthm, amsmath, amssymb}
\usepackage[nobysame]{amsrefs}
\usepackage{enumitem}
\usepackage{setspace}

\SetLabelAlign{CenterWithParen}{\hfil{\upshape(\makebox[1.0em]{#1})}\hfil}

\newtheorem*{hyp*}{Hypothesis \bf{\((\ast )\)}}
\newtheorem{thm}{Theorem}[section]

\newtheorem{lem}[thm]{Lemma}
\newtheorem{cor}[thm]{Corollary}

\newcommand{\irr}[1]{\text{Irr}(#1)}
\newcommand{\cod}[1]{\text{cod}(#1)}
\newcommand{\cd}[1]{\text{cd}(#1)}

\renewcommand{\ker}[1]{\text{ker}(#1)}

\begin{document}

\title[\(p\)-groups with \(p^2\) as a codegree]{\(p\)-groups with \(p^2\) as a codegree} 

\author{Sarah Croome}
\address{%
Department of Mathematical Sciences\\
Kent State University\\
Kent, OH 44242}

\email{scroome@kent.edu}

\author{Mark L. Lewis}
\address{Department of Mathematical Sciences\\
Kent State University\\
Kent, OH 44242}
\email{lewis@math.kent.edu}


\begin{abstract} Let \(G\) be a \(p\)-group and let \(\chi\) be an irreducible character of \(G\). The codegree of \(\chi\) is given by \(|G:\text{ker}(\chi)|/\chi(1)\). This paper investigates the relationship between the nilpotence class of a group and the inclusion of \(p^2\) as a codegree. If \(G\) is a finite \(p\)-group with coclass \(2\) and order at least \(p^5\), or coclass \(3\) and order at least \(p^6\), then \(G\) has \(p^2\) as a codegree. With an additional hypothesis this result can be extended to \(p\)-groups with coclass \(n\ge 3\) and order at least \(p^{2n}\). \end{abstract}

\maketitle

\section{Introduction}
In this paper all groups are finite \(p\)-groups for a prime \(p\), and we examine the relationship between the nilpotence class of a group and the existence of \(p^2\) as a codegree. The codegree of an irreducible character \(\chi\) of a finite group \(G\) is defined as \(|G:\text{ker}(\chi)|/\chi(1)\). The set of codegrees of the irreducible characters of a finite group \(G\) is denoted \(\cod{G}\). This definition for codegrees first appeared in \cite{OG}, where the authors use a graph-theoretic approach to compare the structure of a group with its set of codegrees. More recently, Du and Lewis showed that \(p\)-groups with exactly 3 codegrees have nilpotence class at most \(2\) \cite{codandnil}. In \cite{memax}, it was shown that if \(G\) has order \(p^{n}\) and \(\cod{G}\) contains every power of \(p\) up to \(p^{n-1}\), then \(G\) either has maximal class or nilpotence class at most 2.  The set of codegrees of a \(p\)-group always includes \(p\) \cite[Lemma 2.4]{codandnil}, and Lemma 2.3 of \cite{memax} shows that if \(G\) has maximal class, \(p^2\) is always included in \(\text{cod}(G)\). When \(p^2\) is missing from the set of codegrees of a group \(G\), the quotient \(G/G'\) is elementary abelian \cite[Corollary 2.5]{codandnil}, and, for \(p\)-groups such that \(|G|=p^n\) for \(5\le n\le 7\), the nilpotence class of \(G\) is at most \(n/2\). 

The coclass of a \(p\)-group \(G\) with nilpotence class \(n\) is defined as \(\log_p( |G|)-n\). As \(p\)-groups with large nilpotence class relative to their order have a more predictable structure, it is often possible to characterize groups with small coclass in ways that are impossible for groups with a fixed nilpotence class but arbitrarily large order. The following theorem describes a feature shared by large enough \(p\)-groups of coclass 2 and coclass 3.

\begin{thm}
\label{coclass} Let \(G\) be a \(p\)-group. 
\begin{enumerate}
[label=(\roman*),font=\upshape]
\item If \(G\) has coclass 2 and order at least \(p^5\), then \(p^2\in \cod{G}\). \label{coclass_2}
\item \label{coclass_3} If \(G\) has coclass 3 and order at least \(p^6\), then \(p^2\in\cod{G}\).
\end{enumerate}
\end{thm}

With an additional hypothesis, we can broaden our results to \(p\)-groups of arbitrary finite order which have large enough nilpotence class. This condition was seen to occur in groups included in the Small Groups database of Magma \cite{magma}.

\begin{hyp*} If \(G\) is a \(p\)-group with nilpotence class \(n\) such that \(|G|\ge p^{2n}\), then \(|Z_2(G)|\ne p^2\).
\end{hyp*}

\begin{thm}
\label{p2_coclass_star}
Let a \(p\)-group \(G\) and all of its quotients satisfy Hypothesis \((\ast )\). If \(G\) has coclass \(n\ge 3\) and \(|G|\ge p^{2n}\), then \(p^2\in\cod{G}\).
\end{thm}

We do not know if Hypothesis (\(\ast\)) is needed to prove the conclusion of Theorem \ref{p2_coclass_star}.  We do not have any examples of a \(p\)-group \(G\) with coclass \(n \ge 3\) and \(|G| \ge p^{2n}\) such that \(p^2 \not\in \cod G\).  However, at this time, we do not see how to prove the conclusion of Theorem \ref{p2_coclass_star} without using this hypothesis. We expect this work to appear as part of the first author's Ph.D. dissertation at Kent State University.

\section{Main Results}
Our first lemma can be inferred from \cite{codandnil}.

\begin{lem}
\label{nonlinear}
Let \(G\) be a \(p\)-group with \(p^2\notin\cod{G}\). If \(\chi\) is an irreducible character of \(G\) such that \(\cod{\chi}>p\), then \( \chi\) is non-linear. 
\end{lem}

\begin{proof} 
Let \(\chi\) be an irreducible character of \(G\) such that \(\cod{\chi}=p^a\) for some \(a>2\), and suppose that \(\chi\) is linear.  Then \(\ker{\chi}\ge G' \), so \(G/\ker{\chi}\) is abelian. Since \(\chi\) is a faithful irreducible character of  \(G/\ker{\chi}\), this quotient  must be cyclic, and \(|G/\ker{\chi}|=\chi(1)\cod{\chi}=p^a>p^2\). By Corollary 2.5 of \cite{codandnil}, \(G/G'\) is elementary abelian, which is impossible since \(\ker{\chi}\ge G'\) and \(G/\ker{\chi}\) is cyclic with order greater than \(p^2\). 
\end{proof}

We will also make use of the following lemma.

\begin{lem}
\label{lewis_bound_z}
Let \( G \) be a \(p\)-group with nilpotence class \( n \ge 2 \). Then one of the following occurs:
\begin{enumerate}
[label=(\roman*),font=\upshape]
\item \(G \) has maximal class,
\item \( G \) is extraspecial,
\item \( |Z_{n-1}|\ge p^n \).
\end{enumerate}
\end{lem}

\begin{proof}
Let \(G\) have nilpotence class \(n\) and assume \(|Z_{n-1}|=p^{n-1}\).  Notice that this order is as small as possible, and hence \(|Z_{n-1}/Z_{n-2}|=p\).  Since \(Z_{n-1}/Z_{n-2}=Z(G/Z_{n-2})\), and \(G/Z_{n-2}\) has class 2, we have \(|(G/Z_{n-2})'|=p\), which shows that \(G/Z_{n-2}\) is extraspecial. If \(|Z_{n-2}|= 1\), then \(G\) is extraspecial. If  \(|Z_{n-2}|> 1\), then \(G/Z_{n-2}\) is also capable, as the quotient of \(G/Z_{n-3}\) by its center is isomorphic to \(G/Z_{n-2}\). It is known that a group which is both extraspecial and capable has order \(p^3\) \cite[Cor. 8.2]{capable_extra_special}, hence \(|G/Z_{n-2}|=p^3\), which shows that \(G\) has maximal class. \end{proof}

The next lemma gives our first indication of when \(p^2\) will, or will not, be included among a group's codegrees. 

\begin{lem}
\label{p_squared_iff}
Let \(G\) be a \(p\)-group. Then \(p^2\in\cod{G}\) if and only if either the exponent of \(G/G'\) is at least \(p^2\) or there exists \(N\lhd G\) such that \(G/N\) is extraspecial of order \(p^3\).
\end{lem}

\begin{proof}
If \(G/G'\) has exponent at least \(p^2\), then \(p^2\in\cod{G}\), since otherwise \(G/G'\) is elementary abelian \cite[Corollary 2.5]{codandnil}. If \(N\lhd G\) such that \(G/N\) is extraspecial of order \(p^3\), then since \(G/N\) has nilpotence class 2, there exists \(\chi\in\irr{G/N}\) such that \(\chi(1)=p\). Since \(G'=Z\) and \(|Z|=p\), \(\chi\) must be faithful and hence \(\cod{\chi}=p^2\). 

Now assume \(p^2\in \cod{G}\). Let \(\chi\in \irr{G}\) have codegree \(p^2\). If \(\chi\) is linear, then \(p^2=|G:\ker{\chi}|\) and by Lemma 2.27 of \cite{thebook}, \(G/\ker{\chi}\) is cyclic. Since the kernel of any linear character contains \(G'\), we see that the exponent of \(G/G'\) is at least \(p^2\). Now assume \(\chi\) is not linear. By Lemma 2.1 of \cite{codandnil}, \(\chi(1)=p\). Hence, \(p^3=\cod{\chi}\chi(1)=|G:\ker{\chi}|\), which shows that \(G/\ker{\chi}\) is an extraspecial group of order \(p^3\). 
\end{proof}

We are aware of the existence of groups of order \(p^4\) with class \(2\) which do not have \(p^2\) as a codegree. For a particular example, we have the group listed in the Small Groups database of Magma \cite{magma} as SmallGroup\((3^4,14)\). For an arbitrary prime \(p\), let \(G\cong \langle x,y,z \mid a^p=b^p=c^{p^2}, [a,b]=c^p\rangle \). Here \(G\) is the central product of an extraspecial group of order \(p^3\) and \(\mathbb{Z}_{p^2}\). This group has order \(p^4\), and \(G'\) is the unique normal subgroup of order \(p\). Any non-linear irreducible character must be faithful of degree \(p\), and hence has codegree \(p^3\). As \(G/G'\) is elementary abelian, any linear character must have kernel of order at least \(p^3\), and hence has codegree at most \(p\).

As \(p^2\) is not always a codegree of \(G\) when \(|G|=p^5\) and \(G\) has class 2, it will  be useful to know something about the structure of \(G\) in that case. 

\begin{lem}
\label{p_fifth_class_2}
Let \(G\) be a \(p\)-group with order \(p^5\) and nilpotence class 2. If \(p^2\notin \cod{G}\), then \(\cod{G}=\{1,p,p^3\}\), and \(G\) is either extraspecial or has no faithful irreducible characters.
\end{lem}

\begin{proof}
Let \( \chi\in \irr{G} \) with \(\cod{\chi}>p\), and notice that by Lemma \ref{nonlinear} \(\chi\) is non-linear. Then \(\chi(1)\cod{\chi}\le |G|=p^5\) implies \( p^3 \le \cod{\chi}\le p^4\). If \(\cod{\chi}=p^4\), then \(\chi\) is faithful and \(Z\) is cyclic. Since \(G\) has class 2, by Lemma 2.31 of \cite{thebook} we have \( |G:Z|=\chi(1)^2=p^2\), and hence \(|Z|=p^3\). Also notice that \(G'\) is contained in \(Z\), and \(p^2\notin \cod{G}\) implies \(G'\) is elementary abelian. Since \(Z\) is cyclic and contains the elementary abelian subgroup \(G'\), we have \(|G'| = p\). Then \(G/G'\), which is also elementary abelian, contains the cyclic subgroup \(Z/G'\) with order \(p^2\), which is impossible, so \(p^4\notin \cod{G} \). Since \(G\) is not elementary abelian, \(\cod{G}\ne \{1,p\}\) by Lemma 2.4 of \cite{codandnil}. Hence we must have \(p^3\in\cod{G}\), so \(\cod{G}=\{1,p,p^3\}\). 

The linear characters of \(G\) are not faithful, as \(G\) has nilpotence class 2 and hence \(|G'|>1\). Suppose \(\chi \in \irr{G}\) is faithful and note that \(\cod{\chi}=p^3\). Then \(\chi(1)\cod{\chi}=|G|\) implies \(\chi(1) =p^2\) and \(|G:Z|=p^4\). Since \(\langle 1\rangle <G'\le Z\), we have \(G'=Z\) and \(G\) is extraspecial. 
\end{proof}

Both cases of Lemma \ref{p_fifth_class_2} can occur: the extraspecial groups are well known, and the groups identified by Magma \cite{magma} as SmallGroup\((3^5,i)\) for \(i=44,45,64,65,\) and \(66\) offer particular examples with no faithful irreducible characters. In general, consider \(H\cong G\times \mathbb{Z}_p\) where \(G\) is the central product described in the discussion preceding Lemma \ref{p_fifth_class_2}. This group has order \(p^5\), and no faithful irreducible characters. Nonlinear irreducible characters will have degree \(p\) and kernel of size \(p\), giving \(p^3\) as a codegree. Linear characters will have kernel of size at least \(p^4\), as \(H/H'\) is elementary abelian. 

The following is Lemma 2.3 of \cite{memax}, which will be needed for several of the remaining results. This lemma follows from the fact that a maximal class \(p\)-group has a quotient which is extraspecial of order \(p^3\). This quotient has a faithful non-linear irreducible character of degree \(p\), and the codegree of this character is \(p^2\).

\begin{lem}
\label{maxclass}
If \(G\) is a \(p\)-group that has maximal class, then \(p^2\in\cod{G}\).
\end{lem}

Lemma \ref{p_squared_in_p_fifth_class_3} is the next step toward investigating the connection between the order of a \(p\)-group, it's nilpotence class, and the presence of \(p^2\) as a codegree. 

\begin{lem} 
\label{p_squared_in_p_fifth_class_3}
If \( G \) is a group with order \(p^5\) and nilpotence class at least \( 3 \), then \( p^2 \in \text{cod} (G) \). 
\end{lem}

\begin{proof} If \(G\) has nilpotence class 4, then by Lemma \ref{maxclass}, \(p^2\in\cod{G}\). Thus we may assume \(G\) has class 3. Suppose \( p^2\notin \text {cod}(G) \). Since \( G \) does not have maximal class and is not extraspecial, we know by Lemma \ref{lewis_bound_z} that \( |Z_2|= p^3 \). Suppose \( |Z|=p^2 \).  Let \( \chi \in \text {Irr}(G/Z) \) be non-linear. By Lemma 2.1 of \cite{codandnil},  \( \chi(1)<\text{cod}(\chi) \), which implies \(\text{cod}(\chi) \ge p^3\). Thus \( p^4 \le \chi(1)\text{cod}(\chi)=|G:\text{ker}(\chi)| \le |G:Z|=p^3\),
which is impossible. Therefore \( |Z|=p\), and \( p^4\le \chi(1)\text{cod}(\chi) = |G:\text{ker}(\chi)|\le |G:Z|= p^4 \) shows that \(\chi\) is faithful and hence \( Z_2/Z \) is cyclic. As \( |Z|=p \), we have \( Z=[G',G] < G' \le Z_2 \).  Since \( Z_2/Z \) is cyclic, while \(G'/[G',G]\) is elementary abelian, we have \( |G'|=p^2 \). 

Let \( Z_2=\langle a,Z \rangle \). For any \( g\in G \), \([a,g]\in Z \), so 
\(1 =  [a,g]^p=[a^p,g]\), which implies \( a^p \in Z \). As \( G'=\langle a^p, Z \rangle \), this implies \(G'=Z \), a contradiction. Hence \(p^2\in\cod{G}\).
\end{proof}

We can now prove the first half of Theorem \ref{coclass} using induction, with Lemma \ref{p_squared_in_p_fifth_class_3} as the base case.


\begin{proof}[Proof of Theorem \ref{coclass} \ref{coclass_2}]
Induct on \(|G|\). Lemma \ref{p_squared_in_p_fifth_class_3} establishes the base case where \(|G|=p^5\), so assume \(|G|=p^{n+2}\). Since \(G\) has coclass 2, the nilpotence class of \(G\) is \(n\), the class of \(G/Z\) is \(n-1\),  and \(Z\) has order at most \(p^2\). If \(|Z|=p^2\), then \(|G/Z|=p^n\) and by  Lemma \ref{maxclass}, \(p^2\in\cod{G/Z}\). If \(|Z|=p\), then \(|G/Z|=p^{n+1}\), so \(G/Z\) has coclass 2 and by the inductive hypothesis, \(p^2\in\cod{G/Z}\).
\end{proof}

If we increase \(|G|\) in Lemma \ref{p_squared_in_p_fifth_class_3} to \(p^6\), the result \(p^2\in\cod{G}\) will still hold. There are examples of groups of order \(p^6\) with class 2 where \(p^2\notin\cod{G}\), e.g. semi-extraspecial groups. In these groups, \(G'=Z\) and \(|G'|^2\le |G:G'|\) \cite{beisiegel}. Thus if \(|G|=p^6\) and \(G\) has class 2, we have  \(|Z|=|G'|=p^2\). In \cite{verardi}, it is noted that \(G/G'\) is elementary abelian. If a linear character \(\lambda\in\irr{G}\) has codegree 2, then \(G/\ker{\lambda}\) is  a cyclic quotient of order \(p^2\), which is impossible since the kernel of a linear character contains \(G'\). If \(\chi\in\irr{G}\) is nonlinear, then \(\chi(1)=p^2\), (see, for example, \cite{extreme_degrees}), so \(\cod{\chi}>p^2\), and we have \(p^2\notin\cod{G}\). 

The proof of Lemma \ref{p_squared_in_p_sixth} makes use of characterizations found in \cite{extreme_degrees}, called the strong and weak conditions. If for any \(N\unlhd G\), either \(G'\le N\) or \(N\le Z\), a \(p\)-group \(G\) is said to satisfy the strong condition. If we replace the requirement that \(N\le Z\) with \(|NZ:Z|\le p\), \(G\) is said to satisfy the weak condition. 

\begin{lem}
\label{p_squared_in_p_sixth}
If \( G \) is a \( p \)-group with \(|G|=p^6\) and nilpotence class at least  3, then \( p^2 \in \text{cod} (G) \).  
\end{lem}

\begin{proof} 
If \(G\) has nilpotence class \(4\) or \(5\), we have \(p^2\in\cod{G}\) by Theorem \ref{coclass} \ref{coclass_2} and Lemma \ref{maxclass}, respectively. Thus we may assume \(G\) has class 3, and suppose \( p^2 \notin \text{cod}(G) \). The possible codegrees of non-linear characters of \( G \) are \(p^3\), \(p^4\), and \(p^5 \). If \(\varphi\in \text{Irr}(G) \) has codegree \(p^5\), then \(p^6 \le \varphi(1)\text{cod}(\varphi) = |G:\text{ker}(\varphi)| \le |G| = p^6 \), so \( \varphi \) is faithful. If \(\mu \in \text{Irr}(G) \) has codegree \( p^4 \), then \( p^5 = \mu(1)\text{cod}(\mu) = |G:\text{ker}(\mu)| \le p^6 \), so \( \mu \) is faithful or \( |\ker{\mu} |=p \). In the latter case, since \( G \) has class 3, we have \( |G'|\ge p^2 \), and hence \( G/\ker{\mu}\) is nonabelian with class 2 or class 3. If \( G/\ker{\mu}\) has class 3, then Lemma \ref{p_squared_in_p_fifth_class_3} implies \( p^2 \in \cod{G/\ker{\mu}}\), contradicting \( p^2 \notin \cod{G}  \). On the other hand, if \(G/\ker{\mu} \) has class 2, then \(p^4\in \cod{G/\ker{\mu}} \) is impossible by Lemma \ref{p_fifth_class_2}. Thus \(\mu \) must be faithful. At least one of \(p^4\) and \(p^5\) must in \(\cod{G}\), as \(|\cod{G}|\ge 4\) by Theorem 1.2 of \cite{codandnil}. Thus, \(G\) has a faithful irreducible character, and hence \(Z\) is cyclic.

Since \(G\) has class 3, we know that \(G'\not\le Z\), so \(Z\) cannot be realized as the intersection of kernels of only linear characters of \(G\). Therefore \(Z\) must be contained in the kernel of one or more non-linear irreducible characters of \(G\). Since such a character is  clearly not faithful, it must have codegree \(p^3\). Let \(\chi\in\irr{G}\) be one such character. Then \( p^4 \le \chi(1)\cod{\chi} = |G:\ker{\chi}|\le |G:Z| \le p^5 \). which shows that \(|Z|=p\) or \(p^2 \). 
\vspace{3mm}

\emph{Case 1.} Assume \(|Z|=p^2\). Put \(K=\ker{\chi}\) and notice that by the above inequality we now have \(K=Z\). Since \(G_3\) is elementary abelian and contained in \(Z\), which is cyclic, we have \(|G_3|=p\). By Lemmas 2.27 (f) and 2.31 of \cite{thebook}, \(Z(\chi)=Z_2\), and \(|Z_2|=p^4\). 

Let \(Z(G/G_3)=X/G_3\) and notice that \(Z\le X \le Z_2 \). If \(G/G_3\) is extraspecial, then \(X/G_3=(G/G_3)'=G'/G_3\) implies that \(X=G'\) has order \(p^2\), and thus \(X=Z\). This is impossible as \(G\) has class 3 and hence \(G'\neq Z\). By Lemma \ref{p_fifth_class_2}, we now have that \(G/G_3\) has no faithful irreducible characters, thus \(X/G_3\) cannot be cyclic, so \(|X/G_3|=p^2 \) or \(p^3\). 

If \(|X/G_3|=p^3\), then \(X=Z_2\), and \(|G:X|=p^2\) shows that \(G/G_3\) has exactly two noncentral generators. Thus \(|G'|=p^2\). Let \(Z_2=\langle a,z\rangle \) where  \(Z=\langle z \rangle\), and notice that \(a\) and \(z\) each have order \(p^2\). Since \(G/G'\) is elementary abelian, \(g^p\in G'\) for every \(g\in G\). Hence \(G'=\langle a^p , z^p \rangle \). As \(a^p\notin Z\), there exists some \(g\in G \) such that \(1\ne [a^p,g]\), and since \(a\in Z_2\), we have \([a,g]\in Z\) which implies \([a^p,g]=[a,g]^p\). Thus \([a,g]^p\neq 1\), so \([a,g]\) is an element of \(Z\) with order greater than \(p\) and therefore generates \(Z\). Since \([a,g]\) is also an element of \(G'\), this shows that \(Z\le G'\), which is impossible as they have the same order but cannot be equal. Thus we may assume that \(|X/G_3|=p^2\). 

By \cite[Theorem 2.4]{extreme_degrees}, we know that \(|G'/G_3|\ne p\), so \(G'=X\). Recall that since \(Z\) is cyclic, \(G_3\) is the unique normal subgroup of \(G\) with order \(p\), and hence the kernel of any nonfaithful irreducible character of \(G\) must contain \(G_3\). Any such character which is also nonlinear must have codegree \(p^3\), and since \(G/G_3\) has no faithful irreducible characters, its kernel must have order \(p^2\). Thus any nontrivial normal subgroup of G which does not contain \(G'\) is either the kernel of an irreducible character of \(G\) with codegree \(p^3\), having order \(p^2\) and containing \(G_3\), or an intersection of such kernels and therefore equaling \(G_3\). Hence \(G\) satisfies the weak condition, and by Theorem 5.2 of  \cite{extreme_degrees}, \(Z_2/Z\) cannot be cyclic of order \(p^2\). This is a contradiction since \(Z(\chi)/K=Z_2/Z\) is cyclic by Lemma 2.27 (d) of \cite{thebook}. 

\vspace{3mm}
\emph{Case 2.} Assume \( |Z|=p \). Let \( \chi\in\irr{G} \) have codegree \( p^3 \) and put \(K=\ker{\chi}\). If \(|K|=p \) then by Lemma \ref{p_fifth_class_2}, \(G/Z\) is extraspecial. The only capable extraspecial group has order \(p^3 \)  \cite[Cor. 8.2]{capable_extra_special}, so this is impossible, and we may assume that \(|K|=p^2 \). By Lemma \ref{lewis_bound_z}, we have \( |Z_2|\ge p^3 \). Suppose \(|Z_2|=p^4 \). Since \(Z(\chi)\ge Z_2 \), Corollary 2.30 of \cite{thebook} implies \(Z(\chi)=Z_2\).  Since \(|K:Z|=p\), and \(K/Z\) must intersect nontrivially with \(Z_2/Z\), we have that \(K\le Z_2\).  By Lemma 2.27 (d) of \cite{thebook}, \(Z(\chi)/\text{ker}(\chi)=Z_2/K\) is cyclic. Let \(Z_2=\langle a,K\rangle\). Since \(|Z_2:K|=p^2\), we have \(a^p\notin K> Z\), and hence there exists some \(g\in G\) such that \([a^p,g]\ne 1\). As \(a\in Z_2 \), we have \([a,g]\in Z\), so \(1\ne [a^p,g]=[a,g]^p\in Z\). This is impossible since \([a,g]\in Z\) and \(|Z|=p\), thus we may assume \(|Z_2|=p^3\). 

The order of \(G'\) is now either \(p^2 \) or \(p^3\). Let \(Z(\chi)=\langle a,K \rangle\). Observe that \([Z(\chi),G]\) is contained in both \(K\) and \(G'\). If \(|G'|=p^2\), then \(K\cap G'=Z\), giving \([Z(\chi),G]\le Z\), and hence \([a,g]\in Z \) for all \(g\in G\). Since \(a^p\notin Z\), we can find some \(g\in G\) such that \([a^p,g]\ne 1\). As before, we have \(1\ne [a^p,g]=[a,g]^p\), which is impossible since \(|Z|=p\). Thus \(|G'|=p^3\) and we have \(G'=Z_2\).

Recall that the only non-faithful irreducible characters of \(G\) are either linear, or have codegree \(p^3\) and kernel of order \(p^2\). Again, \(G\) satisfies the weak condition and Theorem C of \cite{extreme_degrees} implies \(\cd{G}=\{1,p^2\}\). Since \(G\) has at least one irreducible character \(\chi\) with codegree \(p^3\) and \(|\ker{\chi}|=p^2\), we must have \(p\in \cd{G}\), which is a contradiction. 
\end{proof}

Lemma \ref{p_squared_in_p_sixth} provides the base case for the induction used to prove the second half of Theorem \ref{coclass}.


\begin{proof}[Proof of Theorem \ref{coclass} \ref{coclass_3}]
Induct on \(|G|\). The base case \(|G|=p^6\) is established by Lemma \ref{p_squared_in_p_sixth}, so we may assume \(|G|=p^{n+3}\). As \(G\) has coclass 3, the order of \(Z\) is at most \(p^3\). When \(|Z|=p^3\), \(G/Z\) has maximal class and \(p^2\in\cod{G/Z}\) by Lemma \ref{maxclass}. When \(|Z|=p^2\), \(G/Z\) has coclass 2 and \(p^2\in\cod{G/Z}\) by Theorem \ref{coclass} \ref{coclass_2}. The final possibility is \(|Z|=p\), in which case \(G/Z\) has coclass 3, and by the inductive hypothesis, \(p^2\in\cod{G/Z}\).
\end{proof}

The following is an easy corollary of Lemma \ref{maxclass} and Theorem \ref{coclass}.

\begin{cor}
\label{p_seventh}
If \(G\) is a group with order \(p^7\) and nilpotence class at least \(4\), then \(p^2\in\cod{G}\).
\end{cor}

The question of whether \(p^2\) is in the set of codegrees for groups of order \(p^8\) with class 4 remains unsettled. If there exists a group \(G\) with \(p^2\notin\cod{G}\), we can make certain claims about the group's structure. These claims are detailed in Lemma \ref{p_squared_not_in_p8}.  

\begin{lem}
\label{p_squared_not_in_p8}
Let \(G\) be a \(p\)-group with nilpotence class 4, \(|G|=p^8\), and \(p^2\notin\cod{G}\). Then the following hold:
\begin{enumerate}
[label=(\roman*),font=\upshape]
\item \(Z=G_4\) is the unique normal subgroup of \(G\) of order \(p\),
\item \(Z_2=G_3\) is the unique normal subgroup of \(G\) of order \(p^2\),
\item \(\cod{G/Z_2}=\{1,p,p^3\}\),
\item either \(|Z_3|=p^4\), \(Z_3=G'\), and \(\cd{G/Z_2}=\{1,p,p^2\}\), or \(|Z_3|=p^5\), \(p^4\le|G'|\le p^5\), and \(\cd{G/Z_2}=\{1,p\}\).
\end{enumerate}
\end{lem}

\begin{proof}
Let \(G\) be as stated. If \(|Z|\ge p^{2}\), then \(G/Z\) has class 3 and \(p^{4}\le|G/Z|\le p^{6}\). Suppose \(|G/Z|=p^{4}\). Notice that \(G/Z\) has maximal class, and since \(|G/Z:Z(G/Z)|=p^3\), we have \(\cd{G/Z}=\{1,p\}\). By Lemma \ref{maxclass}, \(p^2\in\cod{G/Z}\), which is a contradiction.  If \(p^{5}\le|G/Z|\le p^{6}\), then \(p^{2}\in\text{cod}(G/Z)\) by Lemmas \ref{p_squared_in_p_fifth_class_3} and \ref{p_squared_in_p_sixth}. Hence \(|Z|=|G_{4}|=p\), which proves (i). 

Suppose \(|G_{3}|\ge p^{3}\), and let \(N\) be a normal subgroup of \(G\) of order \(p^2\) such that \(N\lneq G_{3}\). Then \(G/N\) has order \(p^6\) and class \(3\), and by Lemma \ref{p_squared_in_p_sixth},  \(p^2\in\cod{G/N}\), which is a contradiction. Hence \(|G_{3}|=p^{2}\). Let \(N\lhd G\) with \(|N|\ge p^{2}\), \(N\ngeq G_{3}\). Then \(G/N\) has class 3, \(p^{4}\le|G/N|\le p^{6}\), and as before, we have \(p^2\in\cod{G/N}\), a contradiction. Hence we may assume that all normal subgroups of \(G\) with order at least \(p^{2}\) contain \(G_{3}\). 

Suppose \(G/G_{3}\) has a faithful character. Then \(Z(G/G_{3})=X/G_{3}\) is cyclic. Since \(G/G_{3}\) has class \(2\), we have  \(X/G_{3}\ge(G/G_{3})'=G'/G_{3}\). As \(p^{2}\notin\text{cod}(G)\), \(G'/G_{3}\) is elementary abelian, and hence \(|G'/G_{3}|=p\). Since \(G/G'\) is also elementary abelian while \(X/G_{3}\) is cyclic, we have \(|X/G_{3}|=p^{2}\), and \(G'/G_{3}\) is the unique normal subgroup of \(G/G_3\) of order \(p\). Every nontrivial normal subgroup of a \(p\)-group intersects the center nontrivially, and hence contains \(G'/G_{3}\), so \(G/G_{3}\) satisfies the strong condition in \cite{extreme_degrees}. Thus \(\text{cd}(G/G_{3})=\{1,p^{2}\}\) by Theorem B of \cite{extreme_degrees}. Also notice that \(G/G_{3}\) has no nonfaithful nonlinear characters, as each nontrivial kernel \(K/G_{3}\) contains \(G'/G_{3}\), which implies \(G/K\) is abelian and hence has no nonlinear irreducible characters. Thus, no kernel of a nonlinear character of \(G\) can properly contain \(G_{3}\) and the  kernel of any nonfaithful nonlinear irreducible character of \(G\) is either \(G_{3}\) or \(G_{4}\). Any normal subgroup of \(G\) with order at least \(p^3\) must contain \(G'\), which shows that \(G\) satisfies the weak condition, and by Theorem G (ii) of \cite{extreme_degrees}, \(|G|\le p^{6}\). This is a contradiction, and therefore \(G/G_{3}\) cannot have a faithful character. 

Now, the kernel of a nonlinear character of \(G\) cannot have order \(p^{2}\), and \(|G'|\ge p^{3}\). Since every normal subgroup is the intersection of one or more kernels of irreducible characters, in order to realize \(G_3\) as such a kernel or intesection of kernels, there must be some nonlinear \(\chi\in\text{Irr}(G)\) with \(|\text{ker}(\chi)|\ge p^{3}\). Suppose \(|\text{ker}(\chi)|=p^{5}\). Then \(\text{cod}(\chi)\chi(1)=|G|/|\text{ker}(\chi)|=p^{3}\). Since the degree of \(\chi\) is strictly less than its codegree, we have \(\chi(1)=p\) and \(\text{cod}(\chi)=p^{2}\), a contradiction.
Hence \(p^{3}\le|\text{ker}(\chi)|\le p^{4}\). If \(|\text{ker}(\chi)|=p^{4}\), then \(\text{cod}(\chi)\chi(1)=p^{4}\) implies \(\text{cod}(\chi)=p^{3}\). If \(|\text{ker}(\chi)|=p^{3}\), then \(|G/\text{ker}(\chi)|=p^{5}\), class \(2\), and hence \(\text{cod}(\chi)=p^{3}\) by Lemma \ref{p_fifth_class_2}, proving (iii). 

If \(G/Z\) has no faithful characters, then \(\text{cod}(G/Z)=\{1,p,p^{3}\}\), implying \(G/Z\) has nilpotence class at most 2, a contradiction. Hence \(G/Z\) has a faithful character, and \(Z_{2}/Z\) is cyclic. Let \(Z_{2}=\langle a,Z\rangle\) and suppose \(|Z_2|\ge p^3\). Then \(a^{p}\notin Z\), and there exists some \(g\in G\) such that \([a^{p},g]\ne1\). For all \(x\in G\),
\([a,x]\in Z\), so \([a^{p},x]=[a,x]^{p}=1\) (since \(|Z|=p\)). Hence \(1\ne[a^{p},g]=[a,g]^{p}=1\), a contradiction. Thus \(|Z_{2}|=p^{2}\) and hence \(Z_{2}=G_{3}\), proving (ii).

To see (iv),  consider \(|Z_{3}|\). As \(G/Z_{2}\) has no faithful irreducible characters, \(Z_{3}/Z_{2}\) is not cyclic, so \(|Z_{3}|\ge p^{4}\). Suppose \(|Z_{3}|=p^{6}\). Then \(\text{cd}(G/Z_{2})=\{1,p\}\), which implies \(|\text{ker}(\chi)|=p^{4}\) for all nonlinear \(\chi\in\text{Irr}(G)\) such that \(|\text{ker}(\chi)|\ge p^{2}\). By Theorem B of \cite{extreme_degrees}, \(G/Z_{2}\) satisfies the strong condition. Hence \(\text{ker}(\chi)\le Z_{3}\), and since \(Z(\chi)\ge Z_{3}\), we have that  \(Z_{3}/\text{ker}(\chi)\)
is cyclic. Put \(\text{ker}(\chi)=K\),  \(\overline{G}=G/Z\), and \(\overline{Z_{3}}=\langle\overline{a},\overline{K}\rangle\). For all \(\overline{g}\in\overline{G}\), \([\overline{a},\overline{g}]\in\overline{Z_{2}}\). Since \(\overline{a^{p}}\notin\overline{Z_{2}},\) there is some \(\overline{x}\in\overline{G}\) such that \([\overline{a^{p}},\overline{x}]\ne1\). Then \([\overline{a},\overline{x}]^{p}=[\overline{a^{p}},\overline{x}]\ne1\), but this is a contradiction since \(|\overline{Z_{2}}|=p\), and hence
\([\overline{a},\overline{x}]^{p}=1\). Therefore \(|Z_{3}|\ne p^{6}\). 

Suppose \(|Z_{3}|=p^{4}\). If \(|G'|=p^{3}\), then by Lemma 2.5 of \cite{extreme_degrees}, \(G/Z_{2}\) is not capable, a contradiction. Hence \(|G'|=p^{4}\), that is, \(G'=Z_{3}\). By Lemma 1.1 of \cite{PPO1}, none of \(G\), \(G/Z\), or \(G/Z_{2}\) has an abelian subgroup of index \(p\). By Theorem 22.5 of \cite{PPO1}, neither \(\text{cd}(G)\) nor \(\text{cd}(G/Z_{2})\) is \(\{1,p\}\). Hence \(p^{2}\in\text{cd}(G/Z_{2})\), and we have some \(\chi\in\text{Irr}(G)\) with \(|\text{ker}(\chi)|=p^{3}\). 

Suppose \(p\notin\text{cd}(G/Z_{2})\). Then \(|\text{ker}(\chi)|=p^{3}\) for all nonlinear \(\chi\in\text{Irr}(G/Z_{2})\). By Lemma A.6.2 of \cite{PPO1}, \(Z_{3}\ge\text{ker}(\chi)\) for all such \(\chi\), and hence \(G\) is normally constrained, defined in \cite{nc_pgroups} as a \(p\)-group with \(G_i\) as the only normal subgroup of \(G\) of order \(|G_i|\) for every \(i\), \(1\le i\le c(G)\), where \(c(G)\) is the nilpotence class of \(G\). If \(p\) is odd, then by Theorem 3.5 of \cite{nc_pgroups}, \(|G:G'|=p^{4}\) implies \(p^{2}\le|G_{3}:G_{4}|\), a contradiction since \(|G_{3}:G_{4}|=p\). If \(p=2\), then since \(G/Z\) satisfies the weak condition, Theorem F of \cite{extreme_degrees} implies that \(|G:Z_2|=p^3\) or \(p^4\), a contradiction. Hence \(\text{cd}(G/Z_{2})=\{1,p,p^{2}\}\). 

Finally, suppose \(|Z_{3}|=p^{5}\). Recall that for \(\chi\in\text{Irr}(G)\), \(\chi(1)^{2}\le|G:Z|\). Since \(|G:Z_{3}|=p^{3}\), we have \(\text{cd}(G/Z_{2})=\{1,p\}\). By Lemma 2.4 of \cite{extreme_degrees}, \(|G'|\ne p^{3}\), since \(|G:Z_{3}|\)\ is not a square. 
\end{proof}

If we consider only \(p\)-groups satisfying Hypothesis (\(\ast\)), these results can be extended to groups which are arbitrarily large. Theorem \ref{star_class_p2} restates Theorem \ref{p2_coclass_star} in terms of nilpotence class.


\begin{thm}
\label{star_class_p2}
Let a group \(G\) and all of its quotients satisfy Hypothesis \((\ast )\). If \(|G|=p^{2n}\) or \(p^{2n-1}\) where \(n\ge 3\), and the nilpotence class of \(G\) is at least \(n\), then \(p^2\in\cod{G}\).
\end{thm}

\begin{proof}
Induct on \(n\). The base case when \(n=3\) is established by Lemmas \ref{p_squared_in_p_fifth_class_3} and \ref{p_squared_in_p_sixth}. Now assume \(|G|=p^{2n}\) or \(p^{2n-1}\) and \(c(G)=c\ge n\), where \(n\ge 4\). Suppose \(|Z|\ge p^2\). Then \(c(G/Z)=c-1\) and \(|G/Z|\le p^{2n-2}\) or \(p^{2n-3}\). If \(|G/Z|\ge p^5\) then we are done by the inductive assumption. Since \(c-1\ge 3\), we know \(|G/Z|\ge p^4\), and if \(|G/Z|=p^4\) then \(p^2\in \cod{G/Z}\) by Lemma \ref{maxclass}. 

We may now assume \(|Z|=p\). By Hypothesis \(( \ast )\), \(|Z_2:Z|\ge p^2\). Suppose \(Z_2/Z\) has exponent greater than \(p\), and let \(a\in Z_2\) such that \(a^p\notin Z\). There exists \(g\in G\) such that \([a^p,g]\ne 1\). Since \([a,g]\in Z\), we have \(1\neq [a^p,g]=[a,g]^p\), which is trivial, as \(|Z|=p\). This contradiction shows that \(Z_2/Z\) is elementary abelian, and we can find \(N\lhd G\) such that \(Z<N<G\), \(|N:Z|=p\), and \(c(G/N)=c-1\). Now \(|G/N|=p^{2n-2}\) or \(p^{2n-3}\), and we are done by the inductive assumption.
\end{proof}

Theorem \ref{p2_coclass_star} now follows as a corollary of Theorem \ref{star_class_p2}.


\begin{proof}[Proof of Theorem \ref{p2_coclass_star}]
Let \(G\) have coclass \(n\), and \(|G|=p^{2m}\) or \(p^{2m+1}\) where \(m\ge n\). If \(|G|=p^{2m}\), then \(c(G)=2m-n\ge m\). If \(|G|=p^{2m+1}\), then \(c(G)=2m+1-n\ge m\). In either case, we are done by Theorem \ref{star_class_p2}.\end{proof}

\end{document}